\documentclass[12pt,oneside,english]{amsart}
\usepackage[T1]{fontenc}
\usepackage[latin9]{inputenc}
\usepackage[letterpaper]{geometry}
\usepackage{amsthm}
\usepackage{amstext}
\usepackage{amssymb}

\makeatletter
\numberwithin{equation}{section}
\numberwithin{figure}{section}
  \theoremstyle{remark}
  \newtheorem{rem}{Remark}
  \theoremstyle{plain}
  \newtheorem{prop}{Proposition}
  \theoremstyle{definition}
  \newtheorem{defn}{Definition}
\theoremstyle{plain}
\newtheorem{thm}{Theorem}

\def\C {\mathbb C}

\def\B{\mathbb B}

\newcommand{\cl}{\overline}

\newcommand{\norm}[1]{\left\| #1\right\|}

\newcommand{\p}{\partial}

\newcommand{\R}{\mathbb R}

\newcommand{\set}[1]{\left\{ #1\right\}}

\newcommand{\To}{\longrightarrow}

\newcommand{\W}{\Omega}

\newcommand{\z}{\zeta}

\makeatother

\usepackage{babel}

\begin{document}

\title{On the Szeg\H{o} Metric}

\author{David Barrett, Lina Lee}


\thanks{The first author was supported in part by NSF grant number DMS-0901205.}

\begin{abstract}
We introduce a new biholomorphically invariant metric based on Fefferman's invariant Szeg\H{o} kernel and investigate the relation of the new metric to the Bergman and Carathéodory metrics.  A key tool is a new absolutely invariant function assembled from the Szeg\H{o} and Bergman kernels.
\end{abstract}

\maketitle

\section{Introduction}

In this paper we introduce the Szeg\H{o} metric, which is defined
similarly to the Bergman metric using the Szeg\H{o} kernel instead
of the Bergman kernel. The well-known Szeg\H{o} kernel $S(z,\zeta)$
is a reproducing kernel for $H^{2}(\partial\Omega)$ (the
closure in $L^{2}(\partial\Omega)$ of the set of holomorphic functions that are continuous up
to the boundary); thus \[
f(z)=\int_{\partial\Omega}S(z,\zeta)f(\zeta)\, d\sigma_E(\zeta),\quad\forall f\in H^{2}(\partial\Omega)\]
where $\sigma_E$ stands for the Euclidean surface measure on $\partial\Omega$.
The problem with this definition though is that, unlike the volume
measure on $\Omega$, the Euclidean surface measure is not transformed
nicely under a biholomorphic mapping. To resolve this issue, Fefferman
introduced the Fefferman surface area measure, $\sigma_{F}$ (p.\,259 of \cite{Fefferman79}). 

We define the Szeg\H{o} metric using the Szeg\H{o} kernel with respect
to the Fefferman surface area measure. Hence it is invariant under
biholomorphic mappings. 

In section 2, we provide background information on the Fefferman surface
measure and define the Szeg\H{o} metric. In section 3, we introduce
a biholomorphically invariant function $SK_{\Omega}(z,w)$ which serves to compare the Bergman and Szeg\H{o} kernels and then proceed to use this function to derive a number of asymptotic results relating the 
Szeg\H{o} and Bergman metrics.
In section 4, we show that the Szeg\H{o} metric is always greater
than or equal to the Carathéodory metric. In section 5 we show that  there is no universal upper bound or positive lower bound for the ratio of the Szeg\H{o} and Bergman metrics.

\medskip

{\bf Standing assumption.} We assume throughout this paper that $\W=\set{\rho<0}\subset\subset\C^{n}$ is a strongly pseudoconvex
domain with $C^{\infty}$ boundary.  (We note however that the Szeg\H{o} kernel and metric discussed in this paper will be naturally interpretable on many other domains; transformation laws such as
Propositions \ref{prop:Feff-trans},\,\ref{prop:Sz-ker-trans},\,\ref{Sz-met-trans} and Theorem \ref{thm:SK-trans} below
will hold with additional hypotheses on $\Phi$ as needed.)

\section{Background}

Let $H^{2}(\W)$ be the closure in $L^{2}(\p\W)$ of $A(\W)=\mathcal{O}(\W)\cap C(\cl\W)$.  Then there exists
a sesqui-holomorphic Szeg\H{o} kernel $S(z,\cdot)$ such that \begin{equation}
f(z)=\int_{\p\W}S(z,\z)f(\z)\,d\sigma_{F}(\zeta),\quad\forall f\in H^{2}(\W)\label{eq:szegokernel}
\end{equation}
where $\sigma_{F}$ is the Fefferman measure defined as follows:

\begin{align*}
d\sigma_{F}\wedge d\rho&=c_{n}\sqrt[n+1]{-\det\left(\begin{array}{cc}
0 & \rho_{\cl k}\\
\rho_{j} & \rho_{j\cl k}\end{array}\right)_{1\le j,k\le n}}\,dV
\intertext{or equivalently}
d\sigma_{F}&=c_{n}\sqrt[n+1]{-\det\left(\begin{array}{cc}
0 & \rho_{\cl k}\\
\rho_{j} & \rho_{j\cl k}\end{array}\right)_{1\le j,k\le n}}\;\frac{d\sigma_{E}}{\norm{d\rho}},
\end{align*}
where $\sigma_{E}$ is the usual Euclidean surface measure and $\rho_{j}=\frac{\partial\rho}{\partial z_{j}}$, $\rho_{j\overline{k}}=\frac{\partial\rho}{\partial z_{j}\partial\overline{z}_{k}}$.

 Note
that the surface measure $\sigma_{F}$ does not depend on the choice of the defining function $\rho$;  one can check this letting $\tilde{\rho}=h\rho$,
where $h>0$ is a smooth function, and calculating $d\sigma_{F}$
with $\tilde{\rho}$. 

\begin{rem}\label{r:cn}
The constant $c_n$ used above is a dimensional constant which was left unspecified in \cite{Fefferman79} but has been assigned different values later for convenience in different contexts:  for example, $c_{n}=2^{2n/\left(n+1\right)}$
in \cite{Barrett06} and $c_{n}=1$ in \cite{Hirachi04}.
\end{rem}

\begin{prop} \label{prop:Feff-trans}
Let $\Phi:\W_{1}\To\W_{2}$ be a biholomorphic mapping. Then we have\[
\int_{\p\W_{2}}|f|^{2}\,d\sigma_{F}^{\p\W_{2}}=\int_{\p\W_{1}}|f\circ\Phi|^{2}|\det J_{\C}\Phi|^{\frac{2n}{n+1}}\,d\sigma_{F}^{\p\W_{1}},\]
where $\sigma_{F}^{\partial\Omega_{j}}$ denotes the Fefferman measure
on $\Omega_{j}$ for $j=1,2$ and $J_{\mathbb{C}}\Phi$ is the complex
Jacobian matrix of $\Phi$. \end{prop}

\begin{proof}
Recall that $\Phi$ extends to a diffeomorphism between $\cl\W_1$ and $\cl\W_2$ \cite{Fefferman74}.

Let $\Phi:\W_{1}\To\W_{2}$ be a biholomorphic mapping and $\W_{2}=\set{\rho<0}$.
Then we have \[
d\sigma_{F}^{\p\W_{1}}\wedge d(\rho\circ\Phi)=c_{n}\sqrt[n+1]{-\det\left(\begin{array}{cc}
0 & (\rho\circ\Phi)_{\cl k}\\
(\rho\circ\Phi)_{j} & (\rho\circ\Phi)_{j\cl k}\end{array}\right)}\,dV_{\W_{1}}.\]
Since\[
\left(\begin{array}{cc}
0 & (\rho\circ\Phi)_{\cl k}\\
(\rho\circ\Phi)_{j} & (\rho\circ\Phi)_{j\cl k}\end{array}\right)=\left(\begin{array}{cc}
1 & 0\\
0 & J_{\C}\Phi\end{array}\right)\left(\begin{array}{cc}
0 & \rho_{\cl k}\\
\rho_{j} & \rho_{j\cl k}\end{array}\right)\left(\begin{array}{cc}
1 & 0\\
0 & \cl{J_{\C}\Phi}\end{array}\right),\]
we get\[
\det\left(\begin{array}{cc}
0 & (\rho\circ\Phi)_{\cl k}\\
(\rho\circ\Phi)_{j} & (\rho\circ\Phi)_{j\cl k}\end{array}\right)=\det\left(\begin{array}{cc}
0 & \rho_{\cl k}\\
\rho_{j} & \rho_{j\cl k}\end{array}\right)\left|\det J_{\C}\Phi\right|^{2}.\]
Therefore we have\begin{align*}
d\sigma_{F}^{\p\W_{1}}\wedge d(\rho\circ\Phi)
&=c_{n}\left|\det J_{\C}\Phi\right|^{2/(n+1)}\cdot\sqrt[n+1]{-\det\left(\begin{array}{cc}
0 & \rho_{\cl k}\\
\rho_{j} & \rho_{j\cl k}\end{array}\right)}\left|\det J_{\R}\Phi^{-1}\right|^2\,dV_{\W_{2}}\\
&=c_{n}\left|\det J_{\C}\Phi\right|^{2/(n+1)}\cdot\left|\det J_{\C}\Phi\right|^{-2}\cdot\sqrt[n+1]{-\det\left(\begin{array}{cc}
0 & \rho_{\cl k}\\
\rho_{j} & \rho_{j\cl k}\end{array}\right)}\,dV_{\W_{2}}\\
&=|\det J_{\C}\Phi|^{-2n/(n+1)}d\sigma_{F}^{\p\W_{2}}\wedge d\rho
\end{align*}
and it follows that $d\sigma_{F}^{\p\W_{2}}$ pulls back to $|\det J_{\C}\Phi|^{\frac{2n}{n+1}}\,d\sigma_{F}^{\p\W_{1}}$.
\end{proof}

\begin{prop}\label{prop:Sz-ker-trans}
Let $\Phi:\W_{1}\To\W_{2}$, $\W_{1},\W_{2}\subset\C^{n}$ be a biholomorphic
mapping. Assume there exists a well-defined holomorphic branch of  $\left(\det J_{\C}\Phi(z)\right)^{n/(n+1)}$
on $\Omega_{1}$. Then we have 
\begin{equation}
S_{\W_{1}}(z,w)=S_{\W_{2}}(\Phi(z),\Phi(w))(\det J_{\C}\Phi(z))^{n/(n+1)}\left(\cl{\det J_{\C}\Phi(w)}\right)^{n/(n+1)},\label{eq:Szw}\end{equation}
where $S_{\Omega_{j}}(z,w)$ is the Szeg\H{o} kernel on
$\Omega_{j}$ for $j=1,2$. \end{prop}
\begin{proof}
It is obvious that the right hand side of \eqref{eq:Szw} is anti-holomorphic with respect to $w$, so it will suffice to show that it also satisfies
the reproducing property. 

Let $f\in H^{2}(\W_{1})$. Then we get 
\begin{multline*}
\int_{\p\W_{1}}S_{\W_{2}}(\Phi(z),\Phi(w))\left(\det J_{\C}\Phi(z)\right)^{n/(n+1)}\left(\cl{\det J_{\C}\Phi(w)}\right)^{n/(n+1)}f(w)\, d\sigma_{F}^{\p\W_{1}}(w)\\
=\left(\det J_{\C}\Phi(z)\right)^{n/(n+1)}\int_{\p\W_{2}}S_{\W_{2}}(\Phi(z),\tilde{w})\left(\cl{\det J_{\C}\Phi(\Phi^{-1}(\tilde{w})}\right)^{n/(n+1)}\cdot\\
f(\Phi^{-1}(\tilde{w}))|\det J_{\C}\Phi^{-1}(\tilde{w})|^{2n/(n+1)}d\sigma_{F}^{\p\W_{2}}(\tilde{w}).
\end{multline*}
Note that $\left(\cl{\det J_{\C}\Phi(\Phi^{-1}(\tilde{w})}\right)^{n/(n+1)}|\det J_{\C}\Phi^{-1}(\tilde{w})|^{2n/(n+1)}$
is holomorphic with respect to $\tilde{w}$ since we have \begin{align*}
\left(\cl{\det J_{\C}\Phi(\Phi^{-1}(\tilde{w})}\right)&^{n/(n+1)}|\det J_{\C}\Phi^{-1}(\tilde{w})|^{2n/(n+1)}\\
&=\left(\cl{\det J_{\C}\Phi(\Phi^{-1}(\tilde{w})}\right)^{n/(n+1)}|\det J_{\C}\Phi(\Phi^{-1}(\tilde{w}))|^{-2n/(n+1)}\\
&=\left(\det J_{\C}\Phi(\Phi^{-1}(\tilde{w}))\right)^{-n/(n+1)}.\end{align*}
 Hence we obtain
 \begin{align*}
\int_{\p\W_{1}}S_{\W_{2}}(&\Phi(z),\Phi(w))\left(\det J_{\C}\Phi(z)\right)^{n/(n+1)}\left(\cl{\det J_{\C}\Phi(w)}\right)^{n/(n+1)}f(w)\, d\sigma_{F}^{\p\W_{1}}(w)\\
&=\left(\det J_{\C}\Phi(z)\right)^{n/(n+1)}\left(\det J_{\C}\Phi(\Phi^{-1}(\Phi(z)))\right)^{-n/(n+1)}f\left(\Phi^{-1}(\Phi(z))\right)\\
&=f(z)\end{align*} as required.
\end{proof}

\begin{defn}
We define the Szeg\H{o} metric on $\W$ at $z$ in the direction $\xi$,
$F_{S}^{\Omega}(z,\xi)$, as follows:\[
F_{S}^{\Omega}(z,\xi)=\left(\sum_{j,k=1}^{n}\frac{\p^{2}\log S_{\Omega}(z,z)}{\p z_{j}\p\cl z_{k}}\xi_{j}\cl{\xi_{k}}\right)^{1/2}.\]
\end{defn}

\begin{rem}
Note that one can write $S_{\Omega}(z,w)=\sum_{\alpha}\phi_{\alpha}(z)\overline{\phi_{\alpha}(w)}$
where the $\phi_{\alpha}$'s form an orthnormal basis of $H^{2}(\partial\Omega)$.
Hence $S_{\Omega}(z,z)$ is a positive strongly plurisubharmonic function, ensuring that $F_{S}^{\Omega}(z,\xi)$ is a genuine K\"ahler metric.  The orthonormal expansion may also be used to show that 
\begin{equation}\label{eq:S-v-E}
F_{S}^{\Omega}(z,\xi) \ge \gamma_\Omega 
\,|\xi| \text{ for some positive constant } \gamma_\Omega.
\end{equation}

 \end{rem}
 
 \begin{rem} \label{rem:cn-in-metric}
Note that $F_{S}^{\Omega}(z,\xi)$ does not depend on the choice of the dimensional constant $c_n$ discussed in Remark \ref{r:cn}.
\end{rem}
 
\begin{prop} \label{Sz-met-trans}
The Szeg\H{o} metric is invariant under biholomorphic mappings satisfying the hypotheses of Proposition \ref{prop:Sz-ker-trans}, i.e,
if $\Phi:\W_{1}\To\W_{2}$ is such a mapping and $z\in\W_{1}$,
$\xi\in T_{z}\W_{1}$, then \[
F_{S}^{\W_{1}}(z,\xi)=F_{S}^{\W_{2}}(\Phi(z),J_{\mathbb{C}}\Phi(z)\xi).\]
\end{prop}

\begin{proof}
From \eqref{eq:Szw}, we have 
\[
S_{\W_{1}}(z,z)=S_{\W_{2}}(\Phi(z),\Phi(z))\left|\det J_{\C}\Phi(z)\right|^{2n/(n+1)}.\]
Hence we have\[
\log S_{\W_{1}}(z,z)=\log S_{\W_{2}}(\Phi(z),\Phi(z))+\frac{n}{n+1}\left[\log\left(\det J_{\C}\Phi(z)\right)+\log\left(\cl{\det J_{\C}\Phi(z)}\right)\right].\]
Let $\Phi(z)=w$. Then 
\begin{align*}
\sum_{j,k}\frac{\p^{2}\log S_{\W_{1}}(z,z)}{\p z_{j}\p\cl{z_{k}}}\xi_{j}\cl{\xi_{k}}
&=\sum_{j,k}\sum_{l,m}\frac{\p^{2}\log S_{\W_{2}}(\Phi(z),\Phi(z))}{\p w_{l}\p\cl{w_{m}}}\frac{\p w_{l}}{\p z_{j}}\frac{\p\cl{w_{m}}}{\p z_{k}}\xi_{j}\cl\xi_{k}\\
&=\sum_{l,m}\frac{\p^{2}\log S_{\W_{2}}(w,w)}{\p w_{l}\p\cl{w_{m}}}\left(J_{\mathbb{C}}\Phi\left(z\right)\xi\right)_{l}\left(\overline{J_{\mathbb{C}}\Phi\left(z\right)\xi}\right)_{m}.\end{align*}

\end{proof}

\subsection{The Szeg\H{o} metric on the unit ball}

Let $\mathbb{B}^{n}=\left\{ \rho=|z|^{2}-1<0\right\} \subset\mathbb{C}^{n}$.
Then \[
\det\left(\begin{array}{cc}
0 & \rho_{\cl k}\\
\rho_{j} & \rho_{j\cl k}\end{array}\right)=-1\quad\text{on }\p\B^{n}.\]
Hence $d\sigma_{F}^{\partial\mathbb{B}^{n}}=\frac{c_{n}}{2}\, d\sigma_{E}^{\partial\mathbb{B}^{n}}$
for $S=\set{|z|^{2}=1}\subset\C^{n}$ and the Szeg\H{o} kernel for
the unit ball in $\C^{n}$ is given by\begin{equation}
S(z,\z)=\frac{1}{c_{n}}\frac{(n-1)!}{\pi^{n}}\frac{1}{(1-z\cdot\cl\z)^{n}}.\label{eq:szegokernel-1}\end{equation}
One can rewrite \eqref{eq:szegokernel-1} as follows:
\[
\int_{\p\mathbb{B}^{n}}S(z,\z)f(\z)\,d\sigma_{F}(\z)=\int_{\p\mathbb{B}^{n}}\frac{(n-1)!}{\pi^{n}}\frac{1}{(1-z\cdot\cl\z)^{n}}f(\z)\,d\sigma_{E}(\z)=f(z),\quad\forall f\in H^{2}(\p\B^{n}).\]

If we calculate the Szeg\H{o} metric for $\mathbb{B}^{n}$ at the
origin, we get 
\begin{align*}
\log S(z,z)&=\log\left(\frac{(n-1)!}{c_{n}\cdot2\pi^{n}}\right)-n\log(1-|z|^{2}),\\
\intertext{and}
\frac{\p^{2}\log S(z,z)}{\p z_{j}\p\cl z_{k}}\Big|_{z=0} & =\begin{cases}
n\frac{\cl{z_{j}}z_{k}}{(1-|z|^{2})^{2}}\Big|_{z=0}=0, & j\neq k\\
n\frac{|z_{j}|^{2}}{(1-|z|^{2})^{2}}+n\frac{1}{(1-|z|^{2})}\Big|_{z=0}=n, & j=k\end{cases}.
\end{align*}
Hence we have
\begin{equation}
F_{S}^{\B^{n}}(0,\xi)=\sqrt{n}|\xi|.\label{eq:SMball}\end{equation}

\begin{rem}
Note that the Bergman metric on the unit ball in $\C^{n}$ evaluated
at the origin is given as \begin{equation}
F_{B}^{\B^{n}}(0,\xi)=\sqrt{n+1}|\xi|\label{eq:BMball}\end{equation}
and the Kobayashi or Carathéodory metric on the unit ball in $\C^{n}$
at the origin is given as 
\begin{equation}
F_{K}^{\B^{n}}(0,\xi)=F_{C}^{\mathbb{B}^{n}}(0,\xi)=|\xi|.\label{eq:CKMball}\end{equation}

Since all four metrics are invariant under the automorphism group of $\mathbb{B}^{n}$ which acts transitively on $\mathbb{B}^{n}$,  relations between the metrics at the origin will propagate throughout $\mathbb{B}^{n}$.  In particular, 
 from \eqref{eq:SMball}, \eqref{eq:BMball} and \eqref{eq:CKMball} we obtain
 \begin{equation}
F_{S}^{\mathbb{B}^{n}}(z,\xi)
=\sqrt{n}\,F_{C}^{\mathbb{B}^{n}}(z,\xi)
=\sqrt{n\,}F_{K}^{\mathbb{B}^{n}}(z,\xi)
=\sqrt{\tfrac{n}{n+1}}\,F_{B}^{\mathbb{B}^{n}}(z,\xi),
\quad\forall z\in\mathbb{B}^{n}.\label{eq:unitball}\end{equation}

\end{rem}

\section{An invariant function and some boundary asymptotics}
\begin{thm}\label{thm:SK-trans}
Let \begin{equation}
SK_{\W}(z,w)=\frac{S_{\W}(z,w)^{n+1}}{K_{\W}(z,w)^{n}},\label{SK}\end{equation}
where $S_{\W}$ and $K_{\W}$ are the Szeg\H{o} and Bergman kernels
on $\W$. Then $SK_{\W}(z,w)$ is invariant under biholomorphic mappings satisfying the hypotheses of Proposition \ref{prop:Sz-ker-trans}, i.e.,
if $\Phi:\W_{1}\To\W_{2}$ is such a mapping then we have\[
SK_{\W_{1}}(z,w)=SK_{\W_{2}}(\Phi(z),\Phi(w)).\]
\end{thm}

\begin{proof}
It is a well-known fact (see for example section 6.1 in [7]) that \begin{equation}
K_{\W_{1}}(z,w)=\left(\det J_{\C}\Phi(z)\right)K_{\W_{2}}(\Phi(z),\Phi(w))\left(\det\cl{J_{\C}\Phi(w)}\right).\label{615}\end{equation}
Hence from \eqref{eq:Szw} and \eqref{615}, we get\begin{align*}
\frac{S_{\W_{1}}(z,w)^{n+1}}{K_{\W_{1}}(z,w)^{n}} & =\frac{S_{\W_{2}}(\Phi(z),\Phi(w))^{n+1}\left(\det J_{\C}\Phi(z)\right)^{n}\left(\cl{\det J_{\C}\Phi(w)}\right)^{n}}{K_{\W_{2}}(\Phi(z),\Phi(w))^{n}\left(\det J_{\C}\Phi(z)\right)^{n}\left(\det\cl{J_{\C}\Phi(w)}\right)^{n}}\\
 & =\frac{S_{\W_{2}}(\Phi(z),\Phi(w))^{n+1}}{K_{\W_{1}}(\Phi(z),\Phi(w))^{n}}.\end{align*}
\end{proof}

\begin{rem}
One can easily calculate $SK_{\mathbb{B}^{n}}(z,z)$, where
$\mathbb{B}^{n}$ is the unit ball in $\mathbb{C}^{n}$, using \eqref{eq:szegokernel-1}
and the well known formula 
\[
K_{\mathbb{B}^{n}}(z,w)=\frac{n!}{\pi^{n}}\frac{1}{\left(1-z\cdot\overline{w}\right)^{n+1}}\]
for the Bergman kernel on the unit ball to obtain
 \[
SK_{\mathbb{B}^{n}}(z,z)=\frac{1}{c_{n}^{n+1}}\frac{\left(n-1\right)!}{\left(n\pi\right)^{n}},\quad\forall z\in\mathbb{B}^{n}.\]

\end{rem}

For the remainder of this section we assume that 
the defining function
$\rho$ for $\Omega$ has been chosen to satisfy Fefferman's approximate Monge-Amp\`ere equation 
\begin{equation*}
 -\det\left(
 \begin{array}{cc}
0 & \rho_{\cl k}\\
\rho_{j} & \rho_{j\cl k}\end{array}
\right)_{1\le j,k\le n}
= 1 + O\left(|\rho|^{n+1}\right)
\end{equation*}
(see \cite{Fefferman76} -- we could also use the not-completely-smooth exact solution to this equation \cite{ChengYau80, LeeMelrose82}).

We set $r=-\rho$; thus $r>0$ in $\Omega$.

We have the following asymptotic expansions of the Bergman and Szeg\H{o} kernels  (see \cite{Fefferman74,  Graham, Hirachi04} and additional references cited in these papers, but the material we are quoting is set forth especially clearly  in section 1.1 and Lemma 1.2 from \cite{HirachiKomatsuNakazawa}):
\begin{align*}
K_{\Omega}(z,z) & =
\begin{cases}
\frac{n!}{\pi^{n}r^{n+1}}+\frac{\left(n-2\right)!\cdot q_{\Omega}}{ r^{n-1}}+O\left(\frac{1}{r^{n-2}}\right),& n\ge3\\
\frac{2}{\pi^{2}r^{3}}+ 3 \tilde q_{\Omega}\cdot\log r 
+O\left(1\right),&
n=2
\end{cases}
\\
S_{\Omega}(z,z) & =
\begin{cases}
\frac{\left(n-1\right)!}{c_n\pi^{n}r^{n}}+\frac{\left(n-3\right)!\cdot q_{\Omega}}{c_nr^{n-2}}+O\left(\frac{1}{r^{n-3}}\right), & n\ge4\\
\frac{2}{c_3\pi^{3}r^{3}}+\frac{ q_{\Omega}}{c_3r}+O\left(|\log r|\right), & n=3\\
\frac{1}{c_2\pi^{2}r^{2}}
+\mu_1
+\frac{\tilde q_{\Omega}}{c_2}\cdot r\log r
+O\left(r\right), & n=2
\end{cases},\end{align*}
where $\mu_1\in C^\infty(\cl\W)$ and $q_{\Omega}$ and $\tilde q_{\Omega}$ are certain local geometric boundary invariants -- in terms of Moser's normal form 
\cite{ChernMoser74} we have $q_\Omega=\frac{2}{3\pi^n} \left\| A_{2\bar 2}^0 \right\|^2$ for $n\ge3$ and $\tilde q_\Omega=-\frac{8}{\pi^2} A_{4 \bar 4}^0$ for $n=2$. Moreover, $r^{n+1}K_{\Omega}(z,z)\in C^{n+1-\epsilon}\left(\cl{\Omega}\right)$ and $r^n S_{\Omega}(z,z)\in C^{\max\{n,3\}-\epsilon}\left(\cl{\Omega}\right)$ for each $\epsilon>0$.
(The remainder terms are equal to a power of $r$ times a first-degree polynomial in $\log r$ with coefficients in $C^\infty(\cl\W)$; later in this section the remainder terms have a similar structure but with higher degree  in $\log r$.)

Combining these results we obtain the following.

\begin{thm}\label{thm:SK-asymp}  The function $SK_{\Omega}(z,z)$ satisfies
\begin{equation*}
SK_{\Omega}(z,z) \in 
\begin{cases}
C^{n-\epsilon}\left(\cl{\Omega}\right), & n\ge 3 \\
C^{4-\epsilon}\left(\cl{\Omega}\right), & n=2
\end{cases}
\end{equation*}
with asymptotics
\[
SK_{\Omega}(z,z)=\begin{cases}
\frac{\left(n-1\right)!}{c_n^{n+1}\left(n\pi\right)^{n}}+\frac{\left(n-3\right)!\cdot3q_{\Omega}}{c_n^{n+1}n^{n}}r^{2}+O\left(r^{3}\right), & n\ge4\\
\frac{2}{c_3^{4}\left(3\pi\right)^{3}}+\frac{q_{\Omega}}{c_3^{4}\cdot 9}r^{2}+O\left(r^{3}|\log r|\right), & n=3\\
\frac{1}{c_2^3 4\pi^{2}}+ \mu_2\, r^2 + \mu_3\, r^4\log r + \frac{3\pi^2 \tilde q_\Omega^2}{c_2^3 16} r^6 \log^2 r + O(r^6\log r), & n=2
\end{cases}\]
for $z$ close to the boundary, where $\mu_2, \mu_3\in C^\infty(\cl\W)$.
\end{thm}

We will use this result to examine the relation between the Bergman and Szeg\H{o} metrics.  It will be helpful to introduce the quantity
\begin{equation*}
E(z,\xi)=(n+1) \left(F_{S}^{\Omega}(z,\xi)\right)^2
- n \left(F_{B}^{\Omega}(z,\xi)\right)^2.
\end{equation*}

\begin{thm}\label{thm:metric-asympt}  For $n\ge3$ the following hold.
\begin{enumerate}
\item[(a)] $E\in C^{n-2-\epsilon}\left(T\cl{\Omega}\right)$.
\item[(b)]
There are constants $0<m_\Omega<M_\Omega<\infty$ so that
\begin{equation*}
m_\Omega \, F_{S}^{\Omega}(z,\xi) \le  F_{B}^{\Omega}(z,\xi) \le M_\Omega \,F_{S}^{\Omega}(z,\xi)
\end{equation*}
 on $T\Omega$.
\item[(c)]   $E(z,\xi)=0$ when $z\in \partial\Omega$ and $\xi$ lies in the maximal complex subspace of $T_z\partial\Omega$.
\item[(d)] $E(z,\xi)\equiv0$  on all of $T(\partial\Omega)$ if and only if the boundary is locally spherical. 
\item[(e)] If $\Omega$ is simply connected then $E(z,\xi)\equiv0$ on $T\cl{\Omega}$
if and only if $\Omega$ is biholomorphic to the ball. 
\end{enumerate}
\end{thm}

\begin{proof}
We start by noting that
\begin{equation} \label{eq:SK-rem}
E(z,\xi) = 
\sum_{j,k=1}^{n}\frac{\partial^{2}\left(\log SK_{\Omega}(z,z)\right)}{\partial z_{j}\,\partial\overline{z}_{k}}\,\xi_{j}\overline{\xi}_{k}.
\end{equation}
Then (a) follows from the smoothness result in Theorem \ref{thm:SK-asymp}.  Statement (b) then follows from (a) and  \eqref{eq:S-v-E} along with the Bergman version of \eqref{eq:S-v-E}.

For $z\in\partial\Omega$ we use \eqref{eq:SK-rem} and  Theorem \ref{thm:SK-asymp} to conclude that
\begin{equation*}\label{eq:E-bndry}
E(z,\xi) = \frac{6\pi^n q_{\Omega}}{(n-1) (n-2)} 
\sum_{j,k=1}^{n} r_j r_{\overline k}\xi_{j}\overline{\xi}_{k}
\end{equation*}
and thus  $E(z,\xi)=0$ when $\sum_{j=1}^{n} r_j \xi_{j}=0$,  verifying (c).  From the same computation we see that $E$ will vanish on all of $T(\partial\Omega)$ if and only if the invariant $q_\Omega$ vanishes identically, so from Corollary 2.5 in \cite{BurnsShnider77}  it follows that (d) holds.

 The ``if'' half of (e) follows from \eqref{eq:unitball} and the invariance properties.  The ``only if'' half follows from (d) along with Theorem C in \cite{ChernJi}  (see also \cite{Pincuk} and section 8 of \cite{BurnsShnider76}).
\end{proof}

For $n=2$ we have instead the following result.

\begin{thm}\label{thm:metric-asympt-n=2}  
For $n=2$ the following hold.
\begin{itemize}
\item[(a)] $E\in C^{2-\epsilon}\left(T\cl{\Omega}\right)$.
\item[(b)] 
There are constants $0<m_\Omega<M_\Omega<\infty$ so that
\begin{equation*}
m_\Omega \, F_{S}^{\Omega}(z,\xi) \le  F_{B}^{\Omega}(z,\xi) \le M_\Omega \, F_{S}^{\Omega}(z,\xi)
\end{equation*}
 on $T\Omega$.
\item[(c)] If $E\in C^{4}\left(T\cl{\Omega}\right)$ then the boundary is locally spherical. 
\item[(d)] If $\Omega$ is simply connected then $E\in C^{4}\left(T\cl{\Omega}\right)$ if and only if $\Omega$ is biholomorphic to the ball, in which case we in fact have $E(z,\xi)\equiv0$ on $T\cl{\Omega}$.
\end{itemize}
\end{thm}

\begin{proof} We need to explain part (c), everything else falling into place as before.

If  $E\in C^{4}\left(T\cl{\Omega}\right)$ then the $r^6\log^2 r$ term from the expansion in Theorem \ref{thm:SK-asymp} must disappear, forcing $\tilde q_\Omega\equiv 0$.
Using an argument of Burns appearing as Theorem 3.2 in Graham's paper \cite{Graham} along with the previously cited material from \cite{HirachiKomatsuNakazawa} we obtain  revised expansions 
\begin{align*}
K_{\Omega}(z,z) & = \frac{2}{\pi^2 r^3} + \mu_4 + 9\, q^*_\Omega r \log r + O(r)\\
S_{\Omega}(z,z) & = \frac{1}{c_2\pi^2 r^2} + \mu_5 + \frac{1}{c_2}q^*_\Omega r^2 \log r + O(r^2) \\
SK_{\Omega}(z,z) & =\frac{1}{c_2^3 4\pi^2} + \mu_6\,r^2 - \frac{3  q^*_\Omega}{c_2^3 2} \, r^4 \log r  + O(r^4) \\
E(z,\xi) &=  \mu_7 - 72\, q^*_\Omega \, r^2 \log r \sum_{j,k=1}^{n} r_j r_{\overline k}\xi_{j}\overline{\xi}_{k} + O(r^2),
\end{align*} 
where $q^*_\Omega = \frac{8}{15\pi^2} |A_{2\bar 4}^0|^2$ and $\mu_4, \mu_5, \mu_6\in C^\infty(\cl\W), \,\mu_7\in C^\infty(T\cl\W)$.  Our smoothness assumption on $E$ now forces $q^*_\Omega\equiv0$
and this in turn implies that the boundary is spherical.  (We note that by Proposition 1.9 in \cite{Graham}, the condition $\tilde q_\Omega\equiv0$ alone does not guarantee that the boundary is spherical.) \end{proof}

For the sake of completeness we also record the corresponding results in one dimension.

\begin{thm}\label{thm:metric-asympt-n=1}
 For $n=1$ the following hold.
\begin{enumerate}
\item[(a)] $E\in C^{\infty}\left(T\cl{\Omega}\right)$.
\item[(b)]
There are constants $0<m_\Omega<M_\Omega<\infty$ so that
\begin{equation*}
m_\Omega \, F_{S}^{\Omega}(z,\xi) \le  F_{B}^{\Omega}(z,\xi) \le M_\Omega \,F_{S}^{\Omega}(z,\xi)
\end{equation*}
 on $T\Omega$.
\item[(c)]  If $\Omega$ is simply connected then $E(z,\xi)\equiv0$.
\end{enumerate}
\end{thm}

\begin{proof}
(a) follows from Theorem 23.2 in \cite{Bell92} and
the well-known fact that  $r S_{\Omega}(z,z)$ and $r^{2}K_{\Omega}(z,z)$ are in $C^{\infty}\left(\cl{\Omega}\right)$ and are nowhere vanishing on $\cl{\Omega}$. 
Statement (b) follows from (a) as in the proof of Theorem \ref{thm:metric-asympt} above.

(c) follows from \eqref{eq:unitball}, invariance properties and the Riemann mapping theorem.
\end{proof}

\section{Comparison with the Carathéodory metric}

In this section we discuss the comparison between the Carathéodory
and Szeg\H{o} metrics and show that the Szeg\H{o} metric is always
greater than or equal to the Carathéodory metric. The proof follows
the same method that was used to show that the Bergman metric is greater
than or equal to the Carathéodory metric in \cite{Hahn}.

We define the Carathéodory metric on a domain $\Omega\subset\mathbb{C}^{n}$
at $p\in\Omega$ in the direction $\xi\in\mathbb{C}^{n}$, $F_{C}^{\Omega}(p,\xi)$,
as \[
F_{C}^{\Omega}(p,\xi)=\sup\left\{ \left(\sum_{j=1}^{n}\left|\frac{\partial\phi\left(p\right)}{\partial z_{j}}\xi_{j}\right|^{2}\right)^{1/2}:\phi\in\mathcal{O}\left(\Omega,\Delta\right),\,\phi\left(p\right)=0\right\} ,\]
where $\mathcal{O}\left(\Omega,\Delta\right)$ denotes the set of
holomorphic mappings from $\Omega$ to $\Delta$, the unit disc in
$\mathbb{C}$. 
\begin{thm}\label{thm:Sz-v-Cara}
The Szeg\H{o} metric is greater than or equal to the Carathéodory
metric. \end{thm}
\begin{proof}
One can show that \begin{equation}
\left(F_{S}^{\W}(p,\xi)\right)^{2}=\frac{\sup\left\{ |\xi g(p)|^{2}:g\in H^{2}(\p\W),\, g(p)=0,\,\norm g_{L^{2}(\p\W)}=1\right\} }{S(p,p)}\label{eq:metric}\end{equation}
using the Hilbert space method. Refer to Theorem 6.2.5 in \cite{JarnickiPflug}
for further details. 

Let  $p\in\Omega$. We have\[
\left\Vert S(\cdot,p)\right\Vert _{L^{2}\left(\partial\Omega\right)}^{2}=\norm{S(p,\cdot)}_{L^{2}(\partial\W)}^{2}=\int_{\partial\W}\cl{S(p,\z)}S(p,\z)\, d\sigma_{F}(\z)=\cl{S(p,p)}=S(p,p).\]
Let $\phi:\Omega\longrightarrow\Delta$ be a holomorphic function with
$\phi(p)=0$. Define a holomorphic function $g:\Omega\longrightarrow\Delta$
as follows: \[
g(z)=\frac{S(z,p)}{\sqrt{S(p,p)}}\phi(z).\]
Then $\norm g_{L^{2}(\partial\W)}\le1$ and $g(p)=0$. Hence from
\eqref{eq:metric} we get\begin{equation}
\left(F_{S}^{\W}(p,\xi)\right)^{2}\ge\frac{\left|\xi g(p)\right|^{2}}{S(p,p)}=\frac{S(p,p)^{2}|\xi\phi(p)|^{2}}{S(p,p)^{2}}=|\xi\phi(p)|^{2}.\label{eq:612-1}\end{equation}
Therefore we get $F_{S}^{\W}(p,\xi)\ge F_{C}^{\W}(p,\xi)$. \end{proof}

\begin{rem}
This argument works on any smoothly bounded pseudoconvex domain where the Szeg\H{o} metric is defined.
\end{rem}

\begin{rem}
The equation \eqref{eq:unitball} shows that the inequality $F_{S}^{\Omega}(p,\xi)\ge F_{C}^{\Omega}(p,\xi)$
is sharp  even in some cases where $F_{C}^{\Omega}(p,\xi)>0$, whereas
we have $F_{B}^{\Omega}(p,\xi)\gneq F_{C}^{\Omega}(p,\xi)$
if $F_{C}^{\Omega}(p,\xi)>0$ \cite{JarnickiPflug}. 
\end{rem}

\section{Comparison with the Bergman metric}
In this section we carry out some computations on annuli to show that the constants $m_\Omega$  and  $M_\Omega$ in Theorem \ref{thm:metric-asympt-n=1}  must depend on $\Omega$.

\begin{thm}
There are no constants  $0<m<M<\infty$ independent of $\Omega$ with the property that
\begin{equation*}
m \, F_{S}^{\Omega}(z,\xi) \le  F_{B}^{\Omega}(z,\xi) \le M \, F_{S}^{\Omega}(z,\xi)
\end{equation*}
on $T\Omega$.
\end{thm}
\begin{proof}
The results of Proposition \ref{pro:Bergman-Szego-ring} below show that
\begin{align*}
F_{S}^{\Omega_{r}}(\sqrt{r},1) / F_{B}^{\Omega_{r}}(\sqrt{r},1) &\to\infty \\
\intertext{and}
F_{S}^{\Omega_{r}}(\sqrt[5]{r},1) / F_{B}^{\Omega_{r}}(\sqrt[5]{r},1) &\to0\end{align*}
as $r\to0$, where $\Omega_{r}=\left\{ z\in\mathbb{C}:r<\left|z\right|<1\right\} $. \end{proof}

\begin{prop}
\label{pro:Bergman-Szego-ring} Let $\Omega_{r}=\left\{ r<\left|z\right|<1\right\} \subset\mathbb{C}$
and $r\in\left(0,1\right)$. We have
\begin{align*}
\lim_{r\to0}\frac{F_{B}^{\Omega_{r}}(\sqrt{r},1)}{\sqrt{\log(1/r)}}  =2&\quad\mbox{and}\quad\lim_{r\to0}\sqrt{r}\cdot F_{S}^{\Omega_{r}}(\sqrt{r},1)=\frac{1}{2}.\\
\intertext{Also,}
\lim_{r\to0}\frac{F_{B}^{\Omega_{r}}(\sqrt[5]{r},1)}{\sqrt{\log(1/r)}}=\sqrt{2} & \quad\mbox{and}\quad\lim_{r\to0}F_{S}^{\Omega_{r}}(\sqrt[5]{r},1)=1.\end{align*}
\end{prop}

\begin{proof}  On the boundary of a planar domain, the Fefferman measure is $\frac{c_1}{2}\,ds$, where $ds$ denotes the element of arclength. In view of Remark \ref{rem:cn-in-metric}, we may set $c_1=2$ so that 
$d\sigma_{F}=ds$. 

The Szeg\H{o} and Bergman spaces
of $\Omega_{r}$ admit  orthonormal bases
$\left\{ a_{n}(r)z^{n}\right\} _{n\in\mathbb{Z}}$
and $\left\{ b_{n}(r)z^{n}\right\} _{n\in\mathbb{Z}}$ with $ a_{n}(r)$ and $b_{n}(r)\ge0$; thus $B_{r}(z,\zeta)=\sum\limits_{n\in\mathbb{Z}}\left(b_{n}(r)\right)^{2}z^{n}\overline{\zeta^{n}}$
and $S_{r}(z,\zeta)=\sum\limits_{n\in\mathbb{Z}}\left(a_{n}(r)\right)^{2}z^{n}\overline{\zeta^{n}}.$

One can calculate $a_{n}(r)$ and $b_{n}(r)$
as follows: 
\begin{align*}
\int_{\partial\Omega_{r}}|a_{n}(r)z^{n}|^{2}\,ds 
& =\int_{|z|=r}r^{2n}|a_{n}(r)|^{2}\,ds+\int_{|z|=1}|a_{n}(r)|^{2}\,ds\\
 & =|a_{n}(r)|^{2}\,2\pi(r^{2n+1}+1)=1,\end{align*}
hence
\[
|a_{n}(r)|^{2}=\frac{1}{2\pi(1+r^{2n+1})},\quad n\in\mathbb{Z}.\]
Also we have\begin{align*}
\int_{\Omega_{r}}|b_{n}(r)z^{n}|^{2}\,dA & =\int_{0}^{2\pi}\int_{r}^{1}|b_{n}(r)|^{2}t^{2n}t\, dt\,d\theta\\
 & =|b_{n}(r)|^{2}2\pi\frac{1}{2n+2}(1-r^{2n+2})=1,\quad n\neq-1,\\
\int_{\Omega_{r}}|b_{-1}(r)z^{-1}|^{2}\,dA & =\int_{0}^{2\pi}\int_{r}^{1}|b_{-1}(r)|^{2}\frac{1}{t}\, dt\,d\theta=|b_{-1}(r)|^{2}2\pi\ln(1/r)=1,\end{align*}
and so
\begin{align*}
|b_{n}(r)|^{2} & =\begin{cases}
\frac{n+1}{\pi}\cdot\frac{1}{1-r^{2n+2}}, & n\neq-1,\\
\frac{1}{2\pi\ln(1/r)}, & n=-1.\end{cases}\end{align*}

Let $B_{r}(z,\zeta)$ and $S_{r}(z,\zeta)$
be the Bergman and Szeg\H{o} kernel on $\Omega_{r}$ respectively
and $z\in\Omega_{r}$. We have
\begin{align*}
\left(F_{B}^{\Omega_{r}}(z,1)\right)^{2} & =\frac{\partial\overline{\partial}\log B_{r}(z,z)}{\partial z\,\partial\overline{z}}\\
 & =\frac{B_{r}(z,z)\cdot\left(B_{r}(z,z)\right)_{z\overline{z}}-\left|\left(B_{r}(z,z)\right)_{z}\right|^{2}}{\left(B_{r}(z,z)\right)^{2}}\\
 & =\frac{\beta_{0}(z,r)\cdot\beta_{2}(z,r)-\left|\beta_{1}(z,r)\right|^{2}}{\left(\beta_{0}(z,r)\right)^{2}},\end{align*}
where \[
\beta_{0}(z,r)=B_{r}(z,z),\quad\beta_{1}(z,r)=\left(B_{r}(z,z)\right)_{z},\quad\mbox{and}\quad\beta_{2}(z,r)=\left(B_{r}(z,z)\right)_{z\overline{z}}.\]
We also get \[
\left(F_{S}^{\Omega_{r}}(z,1)\right)^{2}=\frac{\alpha_{0}(z,r)\cdot\alpha_{2}(z,r)-\left|\alpha_{1}(z,r)\right|^{2}}{\left(\alpha_{0}(z,r)\right)^{2}},\]
where\[
\alpha_{0}(z,r)=S{}_{r}(z,z),\quad\alpha_{1}(z,r)=\left(S_{r}(z,z)\right)_{z},\quad\mbox{and}\quad\alpha_{2}(z,r)=\left(S_{r}(z,z)\right)_{z\overline{z}}.\]
Let us calculate $\alpha_{j}(r^{q},r)$ for $j=0,1,2$,
$q>0$ and estimate $F_{S}^{\Omega_{r}}(r^{q},1)$:
\begin{align*}
2\pi\,\alpha_{0}(r^{q},r) & =\sum_{n\in\mathbb{Z}}\frac{1}{(1+r^{2n+1})}r^{2nq},\\
2\pi\,\alpha_{1}(r^{q},r) & =\sum_{n\in\mathbb{Z}}\frac{1}{\left(1+r^{2n+1}\right)}\cdot n\cdot r^{\left(2n-1\right)q},\\
2\pi\,\alpha_{2}(r^{q},r) & =\sum_{n\in\mathbb{Z}}\frac{1}{\left(1+r^{2n+1}\right)}\cdot n^{2}\cdot r^{2\left(n-1\right)q}.\end{align*}

Note that
\begin{align*}
2\pi\,\alpha_{0}(\sqrt{r},r) & =\frac{2}{1+\sqrt{r}}+\frac{2r}{1+r^{3}}+O\left(r^{2}\right),\\
2\pi\,\alpha_{1}(\sqrt{r},r) & =\frac{-1}{\sqrt{r}\left(1+r\right)}-\frac{\sqrt{r}}{1+r^{3}}+O\left(r^{3/2}\right),\\
2\pi\alpha_{2}(\sqrt{r},r) & =\frac{1}{r\left(1+r\right)}+\frac{5}{1+r^{3}}+O\left(r\right),\end{align*}
and that
\begin{align*}
2\pi\,\alpha_{0}(\sqrt[5]{r},r) & =\frac{1}{1+r}+\frac{r^{2/5}}{1+r^{3}}+O\left(r^{3/5}\right),\\
2\pi\,\alpha_{1}(\sqrt[5]{r},r) & =\frac{r^{1/5}}{1+r^{3}}-\frac{r^{2/5}}{1+r}+O\left(r^{3/5}\right),\\
2\pi\,\alpha_{2}(\sqrt[5]{r},r) & =\frac{1}{1+r^{3}}+\frac{r^{1/5}}{1+r}+O\left(r^{2/5}\right),\end{align*}
which one can verify easily using the comparison test with the geometric
series. 
Therefore we get 
\[
\lim_{r\to0}r\cdot\left(F_{S}^{\Omega_{r}}(\sqrt{r},1)\right)^{2}=\frac{1}{4},\quad\mbox{and}\quad\lim_{r\to0}\left(F_{S}^{\Omega_{r}}(\sqrt[5]{r},1)\right)^{2}=1.\]

One can calculate $\beta_{j}(r^{q},r)$'s for $j=1,2,3$
and estimate $F_{B}^{\Omega_{r}}(r^{q},1)$ in a similar
way:
\begin{align*}
\pi\,\beta_{0}(r^{q},r) & =\left(\sum_{n\in\mathbb{Z}\setminus\left\{ -1\right\} }\frac{\left(n+1\right)}{1-r^{2n+2}}r^{2nq}\right)+\frac{1}{2r^{2q}\log\left(1/r\right)},\\
\pi\,\beta_{1}(r^{q},r) & =\left(\sum_{n\in\mathbb{Z}\setminus\left\{ -1\right\} }\frac{\left(n+1\right)}{1-r^{2n+2}}\cdot n\cdot r^{\left(2n-1\right)q}\right)-\frac{1}{2r^{3q}\log\left(1/r\right)},\\
\pi\,\beta_{2}(r^{q},r) & =\left(\sum_{n\in\mathbb{Z}\setminus\left\{ -1\right\} }\frac{\left(n+1\right)}{1-r^{2n+2}}n^{2}r^{\left(2n-2\right)q}\right)+\frac{1}{2r^{4q}\log\left(1/r\right)}.\end{align*}

We have
\begin{align*}
\pi\,\beta_{0}(\sqrt{r},r) & =\frac{1}{2r\log\left(1/r\right)}+\frac{2}{1-r^{2}}+O\left(r\right),\\
\pi\,\beta_{1}(\sqrt{r},r) & =-\frac{1}{2r^{3/2}\log\left(1/r\right)}-\frac{2}{\sqrt{r}\left(1-r^{2}\right)}+O\left(\sqrt{r}\right),\\
\pi\,\beta_{2}(\sqrt{r},r) & =\frac{1}{2r^{2}\log\left(1/r\right)}+\frac{4}{r\left(1-r^{2}\right)}+O\left(1\right),\end{align*}
and
\begin{align*}
\pi\,\beta_{0}(\sqrt[5]{r},r) & =\frac{1}{2r^{2/5}\log\left(1/r\right)}+\frac{1}{1-r^{2}}+O\left(r^{2/5}\right),\\
\pi\,\beta_{1}(\sqrt[5]{r},r) & =-\frac{1}{2r^{3/5}\log\left(1/r\right)}+\frac{2r^{1/5}}{1-r^{4}}+O\left(r^{3/5}\right),\\
\pi\,\beta_{2}(\sqrt[5]{r},r) & =\frac{1}{2r^{4/5}\log\left(1/r\right)}+\frac{2}{1-r^{4}}+O\left(r^{2/5}\right).\end{align*}
Therefore we get
\[
\lim_{r\to0}\frac{\left(F_{B}^{\Omega_{r}}(\sqrt{r},1)\right)^{2}}{\log\left(1/r\right)}=4\quad\mbox{and}\quad\lim_{r\to0}\frac{\left(F_{B}^{\Omega_{r}}(\sqrt[5]{r},1\right))^{2}}{\log\left(1/r\right)}=2.\]
\end{proof}

\begin{rem}  We note that the Szeg\H{o} and Bergman kernels of $\Omega_r$ can be written in closed form in terms of elliptic functions (see for example \cite{Burbea78}) though that is not particularly helpful for the computations above. 

\end{rem}

\vspace{0.7in}

David Barrett

Dept. of Mathematics, University of Michigan

Ann Arbor, MI 48109-1043 

USA

barrett@umich.edu

\bigskip{}

Lina Lee

Dept. of Mathematics, University of California Riverside

Riverside, CA 92521

USA

linalee@math.ucr.edu
\end{document}